\documentclass[12pt, twoside]{article}
\usepackage{amsmath,amsthm,amssymb,color}
\usepackage{times}
\usepackage{enumerate}
\usepackage{url}
\usepackage[mathlines]{lineno}
\usepackage{cases}
\usepackage{multirow}
\usepackage{enumerate}
\usepackage{microtype}
\usepackage[american]{babel}
\usepackage{indentfirst}
\usepackage[a4paper, left=2.5cm, right=2.5cm, top=2.5cm, bottom=2.5cm]{geometry}
\usepackage[T1]{fontenc}
\usepackage{enumitem}

\setenumerate [1]{itemsep=0pt, partopsep=0pt, parsep=\parskip, topsep=-5pt}
\setitemize[1]{itemsep=0pt, partopsep=0pt, parsep=\parskip, topsep=-5pt}
\setdescription{itemsep=0pt, partopsep=0pt, parsep=\parskip, topsep=-5pt}

\makeatletter
\def\author#1{\gdef\autrun{\def\and{\unskip, }#1}\gdef\@author{#1}}

\makeatother

\newtheorem{theorem}{Theorem}[section]
\newtheorem{lemma}{Lemma}[section]
\newtheorem{definition}{Definition}[section]
\newtheorem{proposition}{Proposition}[section]
\newtheorem{remark}{Remark}[section]

\numberwithin{equation}{section}

\newcommand{\g}{\gamma}
\newcommand{\R}{\mathbf{R}}
\newcommand{\eps}{\varepsilon}

\newcommand{\va}{\varphi}

\begin{document}
	\baselineskip=16pt
	
	\begin{center}
		{\Large \bf The existence and stability of viscosity solutions to perturbed contact Hamilton-Jacobi equations}\\
		\vspace{0.3cm}
		{\sc  Huan Wu$^1$  and Shiqing Zhang$^{2,*}$ }\\[0.3cm]
		
		{\small
			$^1$ School of Mathematics, Sichuan University, Chengdu 610065, China;\\ wuhuan@stu.scu.edu.cn\\
			$^2$ School of Mathematics, Sichuan University, Chengdu 610065, China;\\ zhangshiqing@scu.edu.cn}
	\end{center}
	\footnote[0]{\hspace*{-0.7cm} $^*$Corresponding author.}
	\vspace{-0.5 cm}

	
	\begin{abstract}
		We consider a contact Hamiltonian $H(x,p,u)$  with certain dependence on the contact variable $u$. If $u_{-}$ is a viscosity solution of the contact Hamilton-Jacobi equation
		\[H(x,D_{x}u(x),u(x))=0,\quad x\in M,\]
		and $u_{-}$ is locally Lyapunov asymptotically stable, we will prove that the perturbed equation 
		\[H(x,D_{x}u(x),u(x))+\varepsilon P(x,D_{x}u(x),u(x))=0,\quad x\in M,\]	
	does exist viscosity solution $u_{-}^{\varepsilon}$ which converges uniformly to $u_{-}$, as perturbation parameter $\varepsilon$ converges to 0. Moreover, we give a case that in a neighborhood of  viscosity solution  $u_-$, the perturbed equation has an unique viscosity solution $u_{-}^{\varepsilon}$. Furthermore, $u_{-}^{\varepsilon}$ keeps locally Lyapunov asymptotically stability.

		\textbf{Keywords}: Contact Hamilton-Jacobi equations,  Lyapunov stability, perturbed equations, viscosity solutions.
	\end{abstract}

	\tableofcontents

	
	\section{Introduction}
	
	\subsection { A brief introduction to contact Hamilton-Jacobi equation }
	
	Let $M$ be an n-dimensional connected, closed smooth manifold, with the canonical two-form $\Omega=dx\wedge dp$. $(T^{*}M\times \Omega)$ is symplectic manifold.  The classical Hamilton-Jacobi equation is defined on the symplectic manifold.  The classical Hamilton-Jacobi equation   is a very useful tool to describe the mechanical system due to its geometrical properties,  and it is a bridge between classical mechanics and quantum mechanics. This equation describes isolated and conservative system. While real system constantly interacts with the environment that introduces dissipation and irreversible phenomena. Contact Hamilton-Jacobi equation consists in extending the symplectic phase space of classical mechanics by adding an extra dimension, thus dealing with a contact manifold instead of a symplectic one. The additional dimension is a non-trivial dynamical variable rather than time. Contact Hamilton-Jacobi equation  has been found in many fields, such as the dynamics of dissipative systems and thermodynamic.  We refer to \cite{A89},\cite{AHD17}, \cite{dS17},\cite{dV20},\cite{dV19} for more nice introduction, geometrical description and applications of contact Hamiltonian systems. 
	
	The problems of existence and uniqueness of solutions for the stationary Hamilton-Jacobi equation 
	\begin{equation}\label{eq1.1}\tag{C}
		H(x,D_xu,u)=c, \quad x\in M.
	\end{equation}
	have been raised in past forty years. It is obvious that \eqref{eq1.1} is global nonlinear first-order equation and does not necessarily have classical solutions. Thus these problems have been approached by looking for generalized solutions — usually solutions $u\in W^{1,\infty}_{loc}(M)$, which satisfy the equations almost everywhere, such as \cite{E13}, \cite{SNK75}, \cite{PLL82}. At the beginning of 1980s, M. Crandall and P. L. Lions, in \cite{PLL83}, \cite{PLL84}, introduced a notion of viscosity solution of the equation \eqref{eq1.1} which admits nowhere differentiable functions. And they gave the earliest results for fixed constant $c\in\R$ when $H$ is strictly increasing in $u$. For the case of Hamiltonian $H=H(x,D_xu)$ independent of $u$, Lions and his coauthors   proved the existence of a pair $(u,c)\in C(M)\times\R$ solving the equation \eqref{eq1.1} in \cite{PLL87}.   In the late 90s, G. Contreras and his coauthors in \cite{Cont} pointed out that the constant $c$ in \cite{PLL87} is uniquely determined by the Hamiltonian $H$. Meanwhile,  in \cite{Fathi1}-\cite{Fathi3}, A. Fathi built a connection between the theory of viscosity solution for $u$-independent Hamiltonian equation and Aubry-Mather theory, and constantly improved this theory (i.e., weak KAM theory) under suitable assumptions (\textbf{(H1)-(H2)} listed above). In weak KAM theory, Fathi used a nonlinear solution semi-group $\{T^-_t\}_{t\geqslant0}$, generated by the evolutionary Hamilton equation
	\begin{align*}
		\left\{
		\begin{array}{l}
			\partial_tu+H(x,D_xu)=c, \quad (t,u)\in [0,+\infty)\times M,\\[1mm]
			u(0,x)=\varphi(x)\in C(M),
		\end{array}
		\right.
	\end{align*} 
	to explain that the backward weak KAM solutions are viscosity solutions   for Hamilton equation \eqref{eq1.1} with $u$-independent $H$. For more details about weak KAM theory, readers can refer to \cite{Fathi4}.	
	
	The $u$-dependent problem \eqref{eq1.1} is called contact Hamilton-Jacobi equation. Different from the $u$-independent case, the contact Hamilton-Jacobi equations may admit viscosity solutions for different constant $c$. The authors in \cite{Jing} investigated the existence and uniqueness of solutions $(u,c)\in C(T^n)×\R$ to \eqref{eq1.1} in the viscosity sense and they gave various examples to show the non-uniqueness of $(u,c)$ to the contact equation. For more related literature, please refer to \cite{PLL83}, \cite{PLL84}, \cite{Ishii}. Recently, the authors in \cite{WWY17}, \cite{WWY192} established an implicit variational principle and generalized some fundamental results of weak KAM theory from Hamiltonian systems to contact Hamiltonian systems under Tonelli assumption and Lipschitz continuity $|\frac{\partial H}{\partial u}|\leqslant \lambda$. And in \cite{WY17}, they proved the set of all real numbers $c$'s , for which the equation \eqref{eq1.1} admits viscosity solutions, is an interval $\mathfrak{C}$. Moreover the interval $\mathfrak{C}$ satisfies $(c_l,c_r)\subset\mathfrak{C}\subset[c_l,c_r]$ where the end points $c_l,\ c_r$ have the representation formulas 
	$$c_l:=\inf_{u\in C^{\infty}(M)}\sup_{x\in M}H(x,D_{x}u(x),u(x)), \quad c_r:=\sup_{u\in C^{\infty}(M)}\inf_{x\in M}H(x,D_{x}u(x),u(x)).$$
	So based on above researches, the following admissibility assumption is the premise of our main results.
	
	{\bf Admissibility:} $H(x,p,u)$ is admissible: that is to say, $0\in\mathfrak{C}$, where $\mathfrak{C}$ is the interval mentioned in \cite{WY17}.\\
	\medskip

	\subsection{ main results}
	
	In this paper, let $M$ be a smooth, connected, compact Riemannian manifold without boundary. In the sequel, we choose a $C^{\infty}$ Riemannian metric $g$ on the $M$. We call $d$ the distance on $M$ associated with $g$.   Let $H:T^*M\times\R\to \R$  be a $C^{3}$ contact Hamiltonian which depends on the contact variable $u$, where $T^{*}M$ is the cotangent bundle of $M$. We denote by $(x,p,u)$ the points of  $T^{*}M\times\R$ and always assume the contact Hamiltonian $H(x,p,u)$ satisfies the following conditions 
	\medskip
	\begin{itemize}
		\item [\bf(H1)] Positive Definiteness: For every $(x,p,u)\in T^*M\times\R $, the second partial  derivative $\frac{\partial^2 H}{\partial p^2} (x,p,u)$  is positive definite as a quadratic form;
		\item  [\bf(H2)]Superlinearity: For every $(x,u)\in M\times\R$, $H(x,p,u)$  is superlinear in $p$, i.e., for all $(x,u)\in M\times\R$, constant $A<+\infty$, there exists constant $C(A,x,u)<+\infty$ such that $H(x,p,u)\geqslant A\|p\|-C(A,x,u)$ for all $p\in T^*_xM$;
		\item [\bf(H3)]Lipschitz Continuity: $H(x,p,u)$ is uniformly Lipschitz in $u$, i.e. there exists $\lambda>0$ such that
		$
		\Big|\frac{\partial H}{\partial u}(x,p,u)\Big| \leqslant \lambda,
		$ for any  $(x,p,u)\in T^*M\times\R $.
	\end{itemize}
	\medskip
	
	We consider the following equations 
	\begin{align}\label{0.2}\tag{$\mathrm{HJ_s}$}
		H(x,D_xu(x),u(x))=0,
	\end{align}
	and 
	\begin{equation}\label{eqth1}\tag{$\mathrm{HJ_{\varepsilon}}$}
		H(x,D_xu(x),u(x))+\varepsilon P(x,D_xu(x),u(x))=0, 
	\end{equation}
	where $\varepsilon$ is a small perturbation parameter, $P(x,p,u)$ is a given function of class $C_{0}^3(T^*M\times\R, \R)$, and the symbol $D_{x}$ denotes the gradient with respect to variables $x$. Equation \eqref{eqth1} can be regarded as the $\varepsilon$-perturbed version of equation \eqref{0.2}. Notice that by admissibility condition, we know $H(x,p,u)$ is admissible and \eqref{0.2} has a viscosity solution. We wonder if  $H(x,D_{x}u,u)+\varepsilon P(x,D_{x}u,u)$ is admissible. In this paper we are concerned with viscosity solutions of equation  \eqref{eqth1}.

	The work in this paper is devoted to studying the following two problems:
	
	\textbf{(P1)}: Under what conditions does the equation \eqref{eqth1} has a viscosity solution?

	\textbf{(P2)}: Under what conditions does the viscosity solution of equation \eqref{eqth1} is Lyapunov stable?

	\textbf{A closer look at (P1) and (P2)}.

	We recall that under assumptions \textbf{(H1)-(H3)}, J.Yan, K.Wang and L.Wang \cite{WWY192} provided a representation formula for the  semigroup of operators $\{T^{-}_{t}\}_{t\geqslant 0}:C(M,\R)\to C(M,\R)$, such that for each $\varphi\in C(M,\R)$, $T_{t}^{-}\varphi(x)$ is the unique viscosity solution of the  evolutionary equation 
	\begin{align}\label{CLL}	
		\left\{
		\begin{array}{l}
			u_{t}(t,x)+H(x,D_xu(t,x),u(t,x))=0,\\[1mm]
			u(0,x)=\varphi(x).
		\end{array}
		\right.	
	\end{align}
	$(C(M,\R),\{T_{t}^{-}\}_{t\geqslant 0},\R^+)$ can be viewed as a dynamical system. If there exists $\varphi\in C(M,R)$ such that $T_{t}^{-}\varphi=\varphi,\forall t\geqslant0$, then $\varphi$ is the fixed point of the dynamical system determined by $\{T_{t}^{-}\}_{t\geqslant 0}$. Then it can be proved that $\varphi$ is the stationary viscosity solution of \eqref{CLL}. 
	In particular, See \cite [Proposition 2.7]{WWY19},	we know that a function $u_{-}$ is a viscosity solutions of \eqref{0.2} if and only if $T^{-}_{t}u_{-}=u_{-}$, for any $t\geqslant 0$.  If we denote by $\{T^{\varepsilon}_t\}_{t\geqslant0}$ the semigroup of operators generated by 
	\begin{align}\label{CL}	
		u_{t}(t,x)+H(x,D_xu(t,x),u(t,x))+\varepsilon P(x,D_xu(t,x),u(t,x))=0,
	\end{align}
	Then our problem \textbf{(P1)} is transformed into finding a sufficient condition such that  $T^-_{t}u_-=u_-,\ \forall t\geqslant 0$ for some $u_-\in C(M,\R)$,then which means the existence of the function $u^{\varepsilon}_{-}\in C(M,\R)$ satisfying $T^{\varepsilon}_{t}u^{\varepsilon}_{-}=u^{\varepsilon}_{-}$, for any $t\geqslant 0$. To solve this problem, we will use weak KAM theory and the information of  viscosity solutions   $u_{-}$  of \eqref{0.2}.

	Lyapunov stability is of great importance in the study of dynamical system. For \textbf{(P2)}, we consider the   Lyapunov stability for 
	the fixed points $u^{\varepsilon}_-$ of dynamical systems $(C(M,\R),\{T_{t}^{\varepsilon}\}_{t\geqslant 0},\R^+)$. We assume that the fixed point $u_-$ is Lyapunov stable and wonder whether $u^{\varepsilon}_-$ retain the Lyapunov stability. Let us recall the definition of Lyapunov stability in the setting of this paper. 
	
	\begin{definition}
		Let	$u\in C(M,\R)$ be a  viscosity solution of  \eqref{0.2}. Then
		\begin{itemize}
			\item [(i)] $u$ is called stable (or Lyapunov stable) if for any $\eps>0$ there exists $\delta>0$ such that for any $\varphi\in C(M,\R)$ with $\|\va-u\|_\infty<\delta$ there holds 
			\[
			\|T^-_t\va-u\|_\infty<\eps, \quad \forall t>0.
			\]
			Otherwise, $u$ is unstable (or Lyapunov unstable).
			\item [(ii)] $u$ is called asymptotically stable (or Lyapunov asymptotically stable) if it is stable and there is $\delta>0$ such that 
			\[
			\lim_{t\to+\infty}\|T^-_t\va-u\|_\infty=0
			\]
			for any $\varphi\in C(M,\R)$ satisfying $\|\va-u\|_\infty<\delta$. If $\delta$ can be $+\infty$, we say that $u$ is globally  asymptotically stable.
		\end{itemize}
	\end{definition}
	
	The result of Lyapunov stability problem has been made great progress by many mathematics. Nowadays, this topic is still very active. For dynamical system, Lyapunov's direct method and Lyapunov's indirect method often used to analysis the stability for some solutions. Readers can refer to \cite{H69},\cite{JH69},\cite{NK63},\cite{VZ64}. Since the equation \eqref{eqth1} is not concrete, it's difficult to get Lyapunov function by the methods in the reference mentioned above. Recently, Y. Ruan, K. Wang and J. Yan in \cite{RWY} made a comprehensive analysis for the Lyapunov stability (including asymptotic stability, and instability) of the fixed points within dynamical system $(C(M,\R),\{T_{t}^{-}\}_{t\geqslant 0},\R^+)$ and also prove several uniqueness results for stationary viscosity solutions. Y. Xu, J. Yan and K. Zhao in \cite{Z24} also provided a simple sufficient condition to guarantee that the viscosity solution of equation \eqref{0.2} is Lyapunov asymptotically stable. In the problem \textbf{(P2)}, we focus on Lyapunov stability for stationary viscosity solutions and try to figure out whether this property still holds for the solution of perturbed equation \eqref{eqth1} under the same conditions in \cite{Z24}.  
	
	In this paper, let $\mathcal{S}^-$ and $\mathcal{S}^{\varepsilon}$ denote the set of all viscosity solutions of equation \eqref{0.2} and \eqref{eqth1}, respectively.
	
	The following Theorem \ref{thm1}, as the first main result of this paper, solves the existence of solutions to the perturbed equation. Under the assumptions in Theorem \ref{thm1}, we show that the viscosity solution $u^{\varepsilon}_-$ (gotten in Theorem \ref{thm1}) of perturbed equation \eqref{eqth1} uniformly converges to $u_-$ as $\varepsilon$ tends to $0$.
	\begin{theorem}\label{thm1}
		 Suppose $P\in C_{0}^{3}(T^{*}M\times\R,\R)$ and for all $(x,p,u)\in T^*M\times\R,\ |P(x,p,u)|\leqslant1$,  $u_{-}\in \mathcal{S}^{-}$. If $u_-$ is locally Lyapunov asymptotically stable, then given $\delta>0$, there has $\varepsilon_{\delta}>0$ such that, for all $\varepsilon\in[0,\varepsilon_{\delta}]$, the equation \eqref{eqth1} exists at least one viscosity solution $u^{\varepsilon}_{-}$ satisfying
		$\|u^{\varepsilon}_{-}-u_{-}\|\leqslant\delta$.
	\end{theorem}

	Denote
	\[
	\Lambda_{u_-}:=\operatorname{cl}\Big(\big\{(x,p,u): x \ \text{is a differentiable  point of } \ u_-, p=D_xu_-(x),u=u_-(x)\big\}\Big).
	\]
	
	The following Theorem \ref{thm2}, as the second prime result of this paper, solves the stability of  solutions to perturbed equation. Under the assumptions in Theorem \ref{thm2}, we can get the conclusions of Theorem \ref{thm1} and   the viscosity solution $u^{\varepsilon}_{-}$ is locally Lyapunov asymptotically stable. 
	
	\begin{theorem}\label{thm2}
	 Suppose $P\in C_{0}^{3}(T^{*}M\times\R,\R)$ and for all $(x,p,u)\in T^*M\times\R, \ |P(x,p,u)|\leqslant1\ $, $u_{-}\in \mathcal{S}^{-}$. If for all $(x,p,u)\in\Lambda_{u_-}$, $\frac{\partial H}{\partial u}(x,p,u)>0$,  then we can find a constant $\delta_0>0$ such that for any $\delta\in[0,\delta_{0}]$ and $\varepsilon\in[0,\varepsilon_{\delta}]$, the equation \eqref{eqth1} has a unique viscosity solution $u^{\varepsilon}_{-}$ satisfying
		$$\|u^{\varepsilon}_{-}-u_{-}\|\leqslant\delta,$$
		where $\varepsilon_{\delta}$ is a positive constant depending on $\delta$. Moreover, $u^{\varepsilon}_{-}$ is locally asymptotically stable.
	\end{theorem}



	
	
	\subsection{Notations}
	
	We write as follows a list of symbols used throughout this paper.
	\begin{itemize}

		\item  $D_{x}u(x)=(\frac{\partial u}{\partial x_1},\dots,\frac{\partial u}{\partial x_n})$.
		\item $C^{k}_{0}(T^{*}M\times \R,\R)$ stands for the space of all k-times continuously differentiable  and has compact support functions on $T^{*}M\times \R$.
		\item $\mathcal{S}^{-}$ denotes the set of all backward weak KAM solutions of equation \eqref{0.2}.
		\item $L^\varepsilon$  denotes the Lagrangian associated to $H(x,p,u)+\varepsilon P(x,p,u)$.
		\item  $\phi^{H}_{t}$ denotes the local flow of contact Hamiltonian system \eqref{c}.
		\item $\mathcal{S}^\varepsilon$  denotes the set of all  backward weak KAM solutions of equation \eqref{eqth1}.
		\item $h^{\varepsilon}_{x_0,u_0}(x,t)$ 
		denotes the forward 
		implicit action function associated with $L^{\varepsilon}$.
		\item $\{T^{\varepsilon}_t\}_{t\geqslant 0}$  denotes the backward  solutions semigroup associated with $L^{\varepsilon}$.
	\end{itemize}

	\medskip


	The rest of the paper is organized as follows. Section 2 collects some preliminary results about unperturbed contact Hamiltonian equations and  states some new results about perturbed contact Hamiltonian equations.  Section 3 is devoted to prove Theorem \ref{thm1}, and Section 4  is devoted to prove  Theorem \ref{thm2}.

	\section{ Preliminaries}

	\subsection{The Weak KAM Theorem}
	\begin{itemize}
		\item[$\bullet$]\textbf{ Implicit action function} 
		\medskip
	\end{itemize}

	The Lagrangian  $L(x,\dot{x},u)$ associated to contact Hamiltonian $H(x,p,u)$ is defined as
	\[
	L(x,\dot{x},u):=\sup_{p\in T^*_xM}\{\langle \dot{x},p\rangle-H(x,p,u)\},\quad \forall(x,\dot{x},u)\in TM\times\R,
	\]
	where $\langle \cdot,\cdot\rangle$ represents the canonical pairing between the tangent and cotangent space. By \textbf{(H1)-(H3)}, $L$ satisfies the following properties
	\medskip
	\begin{itemize}
		\item [\bf(L1)]Positive definiteness: For every $(x,\dot{x},u)\in TM\times\R $, the second partial derivative $\frac{\partial^2 L}{\partial \dot{x}^2} (x,\dot{x},u)$ is positive definite as a quadratic form;
		\item  [\bf(L2)]Superlinearity: For each $(x,u)\in M\times\R$, $L(x,\dot{x},u)$ is superlinear in $\dot{x}$, i.e., for all $(x,u)\in M\times\R$, constant $A<+\infty$, there exists constant $C(A,x,u)<+\infty$ such that $L(x,\dot{x},u)\geqslant A\|\dot{x}\|-C(A,x,u)$ for all $\dot{x}\in T_xM$;
		\item [\bf(L3)]Lipschitz continuity: $L(x,\dot{x},u)$ is uniformly Lipschitz in $u$, i.e. there exists $\lambda>0$ such that
		$$
		\Big|\frac{\partial L}{\partial u}(x,\dot{x},u)\Big| \leqslant \lambda,\quad \forall(x,\dot{x},u)\in TM \times\R.	
		$$
	\end{itemize}

	In contact coordinates, the contact Hamiltonian vector field $X_{H}$ can be explicitly written as
	\begin{align}\label{c}\tag{CH}
		\left\{
		\begin{array}{l}
			\dot{x}=\frac{\partial H}{\partial p}(x,p,u),\\[1mm]
			\dot{p}=-\frac{\partial H}{\partial x}(x,p,u)-\frac{\partial H}{\partial u}(x,p,u)p,\qquad (x,p,u)\in T^*M\times\mathbf{R},\\[1mm]
			\dot{u}=\frac{\partial H}{\partial p}(x,p,u)\cdot p-H(x,p,u).
		\end{array}
		\right.
	\end{align}	
	Then we can show the implicit variational principle for contact Hamilton's equations \eqref{c} in \cite{WWY17} as following.
	
	\begin{proposition} (Implicit action function \cite[Theorem A]{WWY17}) \label{prop4}
		For any given $x_0\in M$, $u_0\in\R$, there exists  continuous function  $h_{x_0,u_0}(x,t)$  defined on $M\times(0,+\infty)$ satisfying	
		\begin{equation}\label{2-1}
			h_{x_0,u_0}(x,t)=u_0+\inf_{\substack{\gamma(0)=x_{0} ,  \gamma(t)=x}}\int_0^tL\big(\gamma(\tau),\dot{\gamma}(\tau),h_{x_0,u_0}(\gamma(\tau),\tau)\big)d\tau,
		\end{equation}
		where the infimum is taken among the Lipschitz continuous curves $\gamma:[0,t]\rightarrow M$.
		Moreover, the  infimum  in  \eqref{2-1}  can be achieved.
		If $\gamma_1$ is curve achieving the  infimum  in \eqref{2-1}, then $\gamma_1$ is of class $C^1$.
		Let
		\[x(s):=\gamma_{1}(s),\quad u(s):=h_{x_{0},u_{0}}(\gamma_{1}(s),s),\quad  p(s):=\frac{\partial L}{\partial \dot{x}}(\gamma_{1}(s),\dot{\gamma_{1}}(s),u(s)).\]
		Then $(x(s),p(s),u(s))$ satisfies  equations \eqref{c} with $x(0)=x_0,\ x(t)=x,\ \lim_{s\rightarrow0^+}u(s)=u_0$.
	\end{proposition}
	
	In proposition \ref{prop4}, $h_{x_0,u_0}(x,t)$ is the implicit action function and the  curves achieving infimum in (\ref{2-1}) are the minimizer of $h_{x_0,u_0}(x,t)$. In the following, we show some propositions of the implicit action function.
	\begin{proposition}(\cite[Theorem C, Theorem D]{WWY17},\cite[Prposition 3.3]{WWY192})\label{prop3}
		\item [(1)] Given $x_0\in M$, $u_0\in \R$,for all $s,t>0$ and all $x\in M$, we have
		$$h_{x_0,u_0}(x,t+s)=\inf_{y\in M}h_{y,h_{x_0,u_0}(y,t)}(x,s)$$
		 Moreover, the infimum is attained at $y$ if and only if there exists a minimizer $\gamma$ of $h_{x_0,u_0}(x,t+s)$ with $\gamma(t)=y$.
		\item [(2)] The function $(x_0,u_0,x,t)\to h_{x_0,u_0}(x,t)$ is Lipschitz continuous on $M\times[a,b]\times M\times[c,d]$ for all real numbers $a,b,c,d$ with $a<b$ and $0<c<d$.
		\item [(3)] Given $x_0\in M$ and $u_1$, $u_2\in \R$,
		if $u_1<u_2$, then $h_{x_0,u_1}(x,t)<h_{x_0,u_2}(x,t)$, for all $(x,t)\in M\times(0,+\infty)$.
	\end{proposition}

	\begin{itemize}
		\item[$\bullet$] \textbf{  The  semi-groups of operators }
	\end{itemize} 
	
	As a creative applications of the implicit action function, by using semi-groups of operators, in \cite{WWY192} the authors  provided   a representation formula for the solution semi-group of the evolutionary equation \eqref{CLL}.
	
	Let us recall the  semi-groups of operators introduced in \cite{WWY192}. Define a family of nonlinear operators $\{T^-_t\}_{t\geqslant 0}$ from $C(M,\mathbf{R})$ to itself as follows. For each $\varphi\in C(M,\mathbf{R})$, denote by $(x,t)\mapsto T^-_t\varphi(x)$ the unique continuous function on $ (x,t)\in M\times[0,+\infty)$ such that
	\begin{equation}\label{eq:semigroup}
		T^-_t\varphi(x)=\inf_{\gamma}\left\{\varphi(\gamma(0))+\int_0^tL(\gamma(\tau),\dot{\gamma}(\tau),T^-_\tau\varphi(\gamma(\tau)))d\tau\right\},
	\end{equation}
	where the infimum is taken among the absolutely continuous curves $\gamma:[0,t]\to M$ with $\gamma(t)=x$. The infimum in \eqref{eq:semigroup} can be achieved and the curve achieving the infimum is called a minimizer of $T^-_t\varphi(x)$. Minimizers are of class $C^1$. Let $\gamma$ be a minimizer and $x(s):=\gamma(s)$, $u(s):=T_s^-\varphi(x(s))$, $p(s):=\frac{\partial L}{\partial \dot{x}}(x(s),\dot{x}(s),u(s))$.
	Then $(x(s),p(s),u(s))$ satisfies equations \eqref{c} with $x(t)=x$. $\{T^-_t\}_{t\geqslant 0}$ is a semi-group of operators and the function $(x,t)\mapsto T^-_t\varphi(x)$ is a viscosity solution of \eqref{CLL} with initial condition $w(x,0)=\varphi(x)$. Thus, we call $\{T^-_t\}_{t\geqslant 0}$ the backward solution semi-group.

	Here are some properties of $T^{-}_{t}$ on $C(M\times[0,\infty),R)$.
	
 We have $T^{-}_{t+t'}=T^{-}_{t}\circ T^{-}_{t'} $ for each $t\in R ,t'\in R $; $T^{-}_{0}\varphi(x)=\varphi(x)$. For each $t>0$, the function $T^{-}_{t}\varphi$ is locally semi-concave function. The infimum in \eqref{eq:semigroup} can be achieved, that is, for each $\varphi\in C(M,R)$, each $x\in M$ and each $t>0$ we can find a absolutely  continuous curve $\gamma_{x,t}:[0,t]\to M$ with $\gamma_{x,t}(t)=x$  and
	\[T^-_t\varphi(x)=\varphi(\gamma_{x,t}(0))+\int_0^tL(\gamma_{x,t}(\tau),\dot{\gamma_{x,t}}(\tau),T^-_\tau\varphi(\gamma_{x,t}(\tau)))d\tau.\]

	\begin{proposition}(\cite[Theorem 1.1  ]{WWY192} Representation formula)\label{prop01}
		There exists a semi-group of operators $\{T^{-}_{t}\}_{t\geq 0}:C(M,R)\to C(M,R)$, such that for each $\varphi\in C(M,R)$, $T^{-}_{t}\varphi(x)$ is the unique viscosity solution of equation \eqref{CLL}. Furthermore, 
		\[T^{-}_{t}\varphi(x)=\inf_{y\in M}h_{y,
			\varphi(y)}(x,t),\quad \forall (x,t)\in M\times [0,+\infty),\]
		where $h$ is the implicit action function introduced in proposition \ref{prop4}.
	\end{proposition}

	\begin{proposition}(\cite[Proposition 4.2, 4.3, Corollary 4.1]{WWY192})\label{prop01}
		Let $\varphi$, $\psi\in C(M,\R)$.

		\item [(1)] If $\psi\leqslant\varphi$, then $T^{-}_t\psi\leqslant T^{-}_t\varphi$,$\forall t\geqslant 0$.
		\item [(2)] The function $(x,t)\to T^{-}_t\varphi(x)$ is locally Lipschitz on $M\times(0,+\infty)$.
		\item [(3)] $\|T^{-}_t\varphi-T^{-}_t\psi\|_{\infty}\leqslant e^{\lambda t}\cdot\|\varphi-\psi\|_{\infty},\ \forall t\geqslant0$.

		\item [(4)]  $\{T^{-}_t\}_{t\geqslant0}$ are one-parameter semi-groups of operators. For all $x_0$,  $x\in M$, all $u_0\in \R$ and all $s$, $t>0$,
		
		$T^-_sh_{x_0,u_0}(x,t)=h_{x_0,u_0}(x,t+s)$,\quad  $T^-_{t+s}\varphi(x)=\inf_{y\in M}h_{y,T^-_s\varphi(y)}(x,t)$;

	\end{proposition}


	\begin{itemize}
		\item[$\bullet$]
		\textbf{ Backward Weak KAM solutions of equation \eqref{0.2} }
	\end{itemize}

	\begin{definition}
		\label{bwkam}
		A function $u\in C(M,\R)$ is called a backward weak KAM solution of \eqref{0.2}, if
		\begin{itemize}
			\item [(1)] for each continuous piecewise $C^1$ curve $\gamma:[t_1,t_2]\rightarrow M$, we have
			\begin{align}\label{do}
				u(\gamma(t_2))-u(\gamma(t_1))\leqslant\int_{t_1}^{t_2}L(\gamma(s),\dot{\gamma}(s),u(\gamma(s)))ds;
			\end{align}
			\item [(2)] for each $x\in M$, there exists a $C^1$ curve $\gamma:(-\infty,0]\rightarrow M$ with $\gamma(0)=x$ such that
			\begin{align}\label{cali1}
				u(x)-u(\gamma(t))=\int^{0}_{t}L(\gamma(s),\dot{\gamma}(s),u(\gamma(s)))ds, \quad \forall t<0.
			\end{align}
		\end{itemize}
		We say that $u$ in \eqref{do} is  dominated by $L$, denoted by $u\prec L$. We call curves satisfying \eqref{cali1} are $(u,L,0)$-calibrated curves . We use $\mathcal{S}^-$  to denote the set of all backward  weak KAM solutions of \eqref{0.2}.
	\end{definition}
	\begin{remark}\label{10}
		If $u$ is dominated by $L$ and $\gamma:I\to M$	is $(u,L,0)$-calibrated curve, then, for every subinterval $I'\subset I$, the restriction $\gamma|_{I'}$ is also $(u,L,0)$-calibrated curve.
	\end{remark}
	\begin{remark}\label{ll}
		By the definition of dominated functions, it is not hard to show any dominated function  is Lipschitz continuous on $M$. See \cite[Lemma 4.1]{WWY19} for the proof. Hence, backward weak KAM solutions are Lebesgue almost everywhere differentiable.
	\end{remark}	
	
	\begin{proposition}(\cite[Lemma 4.3]{WWY19})\label{diffe}
		Given any $a>0$, let $u\prec L$ and let $\gamma:[-a,a]\rightarrow M$ be a $(u,L,0)$-calibrated curve. Then $u$ is differentiable at $\gamma(0)$.
	\end{proposition}	
	
	The proposition below points out that backward weak KAM solutions are just fixed points of solution semi-groups $\{T^{-}_{t}\}_{t\geqslant 0}$, and that viscosity solutions are the same as backward weak KAM solutions in the setting of this paper.
	
	\begin{proposition}    \cite[Proposition 2.7]{WWY19} \label{prop9}
		The backward weak KAM solutions of \eqref{0.2} are the same as the viscosity solutions of \eqref{0.2}. Moreover,
		$u_-\in \mathcal{S}^-$ if and only if $T^-_tu_-=u_-$ for all $t\geqslant0$.
	\end{proposition}

	\begin{proposition}(\cite[Proposition 4.1]{WWY19})\label{inva}
		Let $u_-\in \mathcal{S}_-$.
		Let $\gamma:[a,b]\rightarrow M$ be a  $(u_-,L,0)$-calibrated curve. Then $u_-$ is differentiable at $\g(s)$ for all $s\in(a,b)$ and 	
		\[
		\Big(\gamma(s),\frac{\partial L}{\partial \dot{x}}(\g(s),\dot{\g}(s),u_-(\g(s))),u_-(\gamma(s))\Big)
		\]
		is a solution of \eqref{c}.
		Moreover,
		\[
		\big(\gamma(t+s),D_{x}u_-(\gamma(t+s)),u_-(\gamma(t+s))\big)=\Phi^H_{s}\big(\gamma(t),D_{x}u_-(\gamma(t)),u_-(\gamma(t))\big), \quad\forall t, \ t+s\in[a,b],
		\]
		and
		\[
		H\big(\gamma(s),D_{x}u_-(\g(s)),u_-(\gamma(s))\big)=0,\quad D_{x}u_-(\g(s))=\frac{\partial L}{\partial \dot{x}}\big(\gamma(s),\dot{\gamma}(s),u_-(\gamma(s))\big).
		\]
	\end{proposition}

	Suppose $u_{-}\in S^{-}$, recall that
	\[
	\Lambda_{u_-}:=\operatorname{cl}\Big(\big\{(x,p,u): x \ \text{is a differentiable point of } \ u_-, p=D_xu_-(x),u=u_-(x)\big\}\Big).
	\]

	\begin{proposition}\label{prop2.7}
		Suppose $u_-\in\mathcal{S}^-$. Then for any $\tilde{c}>0$ and $(x,p,u)\in\Lambda_{u_-}$, there exist a $(u_-,L,0)$-calibrated curve $\g:[-1,0]\to M$ and $t_0\in(-1,0)$ such that
		$$d((x,p,u),(\g(t_0),D_xu_-(\g(t_0)),u_-(\g(t_0))))<\tilde{c}.$$
	\end{proposition}

	\begin{proof}
		From the definition of $\Lambda_{u_-}$, for any $\tilde{c}>0$, we can find a differentiable point $x_0$ of $u_-$ such that
		\begin{equation}\label{eq2.5}
			d((x,p,u),(x_0,D_xu_-(x_0),u_-(x_0)))<\frac{\tilde{c}}{2}.
		\end{equation}
		According to the definition of backward weak KAM solution, there is a $(u_-,L,0)$-calibrated curve $\g:[-1,0]\to M$ satisfying $\g(0)=x_0$. Using Proposition \ref{diffe}, we know that $u_-$ is differentiable at $\{\g(s):s\in(-1,0)\}$. So $u_-$ is differentiable at every point in $(-1,0]$.   Recall that $u_-$ is a semiconcave function, we can get $u_-(\g(s))\in C^1((-1,0],\R)$. Thus
		$$D_xu_-(x_0)=\lim_{t\to0^-}D_xu_-(\g(t)).$$
		Therefore, we can find a $t_0\in(-1,0)$ such that
		\begin{equation}\label{eq2.6}
			d((x_0,D_xu_-(x_0),u_-(x_0)),(\g(t_0),D_xu_-(\g(t_0)),u_-(\g(t_0))))<\frac{\tilde{c}}{2}.
		\end{equation}
		From \eqref{eq2.5} and \eqref{eq2.6}, we get
		\begin{equation*}
			\begin{split}
				&d((x,p,u),(\g(t_0),D_xu_-(\g(t_0)),u_-(\g(t_0))))\\
				\leqslant&d((x,p,u),(x_0,D_xu_-(x_0),u_-(x_0)))+\\
				&d((x_0,D_xu_-(x_0),u_-(x_0)),(\g(t_0),D_xu_-(\g(t_0)),u_-(\g(t_0))))\\
				<&\frac{\tilde{c}}{2}+\frac{\tilde{c}}{2}=\tilde{c}.
			\end{split}
		\end{equation*}
	\end{proof}
	
	\subsection{ Some fundamental results  for equation \eqref{eqth1} }

	In this section, we suppose the function $P\in C_0^3(T^*M\times\R,\R)$ and $|P(x,p,u)|\leqslant1$, $\forall (x,p,u)\in T^*M\times\R$. Denote
	$$H^{\varepsilon}(x,p,u)=H(x,p,u)+\varepsilon P(x,p,u),\quad \forall (x,p,u)\in T^*M\times\R.$$
	The Lagrangian  $L^{\varepsilon}(x,\dot{x},u)$ associated to $H^{\varepsilon}(x,p,u)$ is defined as
	$$L^{\varepsilon}(x,\dot{x},u):=\sup_{p\in T^*_xM}\{\langle \dot{x},p\rangle-H^{\varepsilon}(x,p,u)\},\quad \forall(x,\dot{x},u)\in TM\times\R,$$
	where $\langle \cdot,\cdot\rangle$ represents the canonical pairing between the tangent and cotangent space. Notice that $P\in C_{0}^{3}(T^{*}M\times R,R)$ and $H$ satisfies \textbf{(H1)-(H3)}, we can find a constant $\theta$ such that if $\varepsilon\in[0,\theta]$, then we can see the Hamiltonian $H^{\varepsilon}$ satisfies the conditions \textbf{(H1)-(H3)}  .
	And the Lagrangian $L^{\varepsilon}$ belongs to $C^3(T^*M\times\R,\R)$ and also satisfies \textbf{(L1)-(L3)}. Thus we can see that Proposition \ref{prop4}-\ref{prop2.7} are still true for the systems related to $H^{\varepsilon}$ and $L^{\varepsilon}$. So without any loss of generality, in the following contexts, we suppose $\varepsilon\in[0,\theta]$. Then  the implicit action function related to $L^{\varepsilon}$ is
	\begin{equation*}
		h^{\varepsilon}_{x_0,u_0}(x,t)=u_0+\inf_{\substack{\gamma(0)=x_{0} ,  \gamma(t)=x}}\int_0^tL^{\varepsilon}\big(\gamma(\tau),\dot{\gamma}(\tau),h^{\varepsilon}_{x_0,u_0}(\gamma(\tau),\tau)\big)d\tau,
	\end{equation*}
	and 	the  semi-group of operators  associated to equation  \eqref{CL} is  
	\begin{equation*}
		T^{\varepsilon}_t\varphi(x)=\inf_{\gamma}\left\{\varphi(\gamma(0))+\int_0^tL^{\varepsilon}\big(\gamma(\tau),\dot{\gamma}(\tau),T^{\varepsilon}_\tau\varphi(\gamma(\tau))\big)d\tau\right\}.
	\end{equation*}

	In the following part, we discuss some fundamental properties of Lagrangian $L^{\varepsilon}$ and semi-groups of operators    $T^{\varepsilon}_t$, which are crucial for the proofs of the main results.

	\begin{proposition}\label{prop1}
		For any $\varepsilon>0$, we have
		\begin{equation*}
			\big|L^{\varepsilon}(x,\dot{x},u)-L(x,\dot{x},u)\big|\leq\varepsilon, \quad \forall (x,\dot{x},u)\in TM \times\R.
		\end{equation*}
		Moreover, $\frac{\partial L^{\varepsilon}}{\partial\dot{x}}\to\frac{\partial L}{\partial\dot{x}}$ uniformly on $TM\times\R$ as $\varepsilon\to0$.
	\end{proposition}

	\begin{proof}
		According to the definition of Legendre transform,
		\begin{equation}\label{eqLarag1}
			L(x,\dot{x},u):=\sup_{p\in T^*_xM}\{\langle \dot{x},p\rangle-H(x,p,u)\},\quad \forall(x,\dot{x},u)\in TM\times\R,
		\end{equation}
		and
		\begin{equation}\label{eqLarag2}
			L^{\varepsilon}(x,\dot{x},u):=\sup_{p\in T^*_xM}\{\langle \dot{x},p\rangle-H^{\varepsilon}(x,p,u)\},\quad \forall(x,\dot{x},u)\in TM\times\R.
		\end{equation}
		Thus for any $(x,\dot{x},u)\in TM\times\R$, we get
		\begin{equation*}
			\begin{split}
				&\Big|L^{\varepsilon}(x,\dot{x},u)-L(x,\dot{x},u)\Big|\\
				=&\Big|\sup_{p\in T^*_xM}\big\{\langle \dot{x},p\rangle-H^{\varepsilon}(x,p,u)\big\}-\sup_{p\in T^*_xM}\big\{\langle \dot{x},p\rangle-H(x,p,u)\big\}\Big|\\
				\leqslant&\sup_{p\in T^*_xM}\Big|\langle \dot{x},p\rangle-H^{\varepsilon}(x,p,u)-\langle \dot{x},p\rangle+H(x,p,u)\Big|\\
				=&\sup_{p\in T^*_xM}\varepsilon\Big|P(x,p,u)\Big|\leqslant\varepsilon.
			\end{split}
		\end{equation*}
		
		On the other hand, for any $(x,\dot{x},u)\in TM\times\R$, we suppose the supremum of \eqref{eqLarag1} and \eqref{eqLarag2} are obtained respectively at $\tilde{p}$ and $\tilde{p}_{\varepsilon}$. By analyzing, we can see that
		$$\frac{\partial H}{\partial p}(x,\tilde{p},u)=\dot{x}=\frac{\partial H^{\varepsilon}}{\partial p}(x,\tilde{p}_{\varepsilon},u).$$
		The maps $f_{x,u}: p\to\frac{\partial H}{\partial p}(x,p,u)$ and $f^{\varepsilon}_{x,u}: p\to\frac{\partial H^{\varepsilon}}{\partial p}(x,p,u)$ are diffeomorphisms from $T^*_xM$ to $T_xM$ since $H$ and $H^{\varepsilon}$ satisfy \textbf{(H1)-(H2)}. So we can view $\tilde{p}$, $\tilde{p}_{\varepsilon}$ as functions of $\dot{x}$: 
		$\tilde{p}=(f_{x,u})^{-1}(\dot{x})$ and $\tilde{p}_{\varepsilon}=(f^{\varepsilon}_{x,u})^{-1}(\dot{x})$. Since
		$$\frac{\partial H^{\varepsilon}}{\partial p}=\frac{\partial H}{\partial p}+\varepsilon\cdot\frac{\partial P}{\partial p}\rightrightarrows\frac{\partial H}{\partial p}\ on\ T^*M\times\R, \quad as\ \varepsilon\to0,$$
		we get $\tilde{p}_{\varepsilon}=(f^{\varepsilon}_{x,u})^{-1}(\dot{x})\to\tilde{p}=(f_{x,u})^{-1}(\dot{x})$ on $TM\times\R$ as $\varepsilon\to0$. 
		
		Take the function $F_{\varepsilon}(x,\dot{x},u):=L^{\varepsilon}(x,\dot{x},u)-L(x,\dot{x},u)$. Through calculation, we have
		\begin{equation*}
			\begin{split}
				&F_{\varepsilon}(x,\dot{x},u)=\langle \dot{x},\tilde{p}_{\varepsilon}\rangle-H^{\varepsilon}(x,\tilde{p}_{\varepsilon},u)-\langle \dot{x},\tilde{p}\rangle+H(x,\tilde{p},u)\\
				=&\langle \dot{x},(f^{\varepsilon}_{x,u})^{-1}(\dot{x})\rangle-H^{\varepsilon}(x,(f^{\varepsilon}_{x,u})^{-1}(\dot{x}),u)-\langle \dot{x},(f_{x,u})^{-1}(\dot{x})\rangle+H(x,(f_{x,u})^{-1}(\dot{x}),u),
			\end{split}
		\end{equation*}
		and
		$$\frac{\partial F_{\varepsilon}}{\partial\dot{x}}(x,\dot{x},u)=(f^{\varepsilon}_{x,u})^{-1}(\dot{x})-(f_{x,u})^{-1}(\dot{x})=\tilde{p}_{\varepsilon}-\tilde{p}.$$
		Therefore, we have $$\frac{\partial F_{\varepsilon}}{\partial\dot{x}}(x,\dot{x},u)\to0, \quad on\ TM\times\R,$$
		which means $\frac{\partial L^{\varepsilon}}{\partial\dot{x}}$ tends to $\frac{\partial L}{\partial\dot{x}}$ uniformly on $TM\times\R$ as $\varepsilon\to0$.
	\end{proof}

	\begin{proposition}\label{prop2}
		Let $\varphi(x)\in C(M,\R)$. For all $\varepsilon>0$, we have
		\begin{equation*}
			\Big|T^{\varepsilon}_{t}\varphi(x)-T^{-}_{t}\varphi(x)\Big|\leq \frac{\varepsilon}{\lambda}(e^{\lambda t}-1),\quad \forall t>0,\ \forall x\in M.
		\end{equation*}
	\end{proposition}

	\begin{proof}
		In fact, for all $t>0$, $x\in M$, there are three situations:
		\begin{itemize}
			\item [\bf(1)] If $T^{\varepsilon}_{t}\varphi(x)= T^{-}_{t}\varphi(x)$, the result has been proven;
			\item [\bf(2)] If $T^{\varepsilon}_{t}\varphi(x)> T^{-}_{t}\varphi(x)$, we need to prove
			\[T^{\varepsilon}_{t}\varphi(x)-T^{-}_{t}\varphi(x)\leq \frac{\varepsilon}{\lambda}(e^{\lambda t}-1);\]
			\item [\bf(3)] If $T^{\varepsilon}_{t}\varphi(x)< T^{-}_{t}\varphi(x)$, we need to prove
			\[T^{-}_{t}\varphi(x)-T^{\varepsilon}_{t}\varphi(x)\leq \frac{\varepsilon}{\lambda}(e^{\lambda t}-1).\]
		\end{itemize}
		We only give a proof of the second situation, since the third one can be shown in a symmetric way.
		
		If $T^{\varepsilon}_{t}\varphi(x)> T^{-}_{t}\varphi(x)$, let $\gamma:[0,t]\to M$ be a minimizer of $T^{-}_{t}\varphi(x)$ satisfying $\gamma(t)=x$. Let $F(s):=T^{\varepsilon}_s\varphi(\gamma(s))-T^-_s\varphi(\gamma(s))$, $s\in(0,t]$. From part (2) of  Proposition \ref{prop01} , we know $F(s)$ is continuous in $(0,t]$.  $F(t)=T^{\varepsilon}_t\varphi(x)-T^-_t\varphi(x)>0$ and $$F(0):=\lim_{s\rightarrow 0}[T^{\varepsilon}_s\varphi(\gamma(s))-T^-_s\varphi(\gamma(s))]=\varphi(\gamma(0))-\varphi(\gamma(0))=0.$$
		Hence, there exists $s_{0}\in [0,t)$ such that $F(s_{0})=T^{\varepsilon}_{s_0}\varphi(\gamma(s_0))-T^-_{s_0}\varphi(\gamma(s_0))=0$ and $F(s)>0$ for $ s\in(s_{0},t]$. Since $\gamma$ is a minimizer of $T^-_t\varphi(x)$, for any $s\in(s_0,t]$, we have
		\begin{equation*}
			T^-_s\varphi(\gamma(s))=T^-_{s-s_0}\circ T^-_{s_0}\varphi(\gamma(s))=T^-_{s_0}\varphi(\gamma(s_0))+\int_{s_0}^sL\big(\gamma(\tau),\dot{\gamma}(\tau),T^-_\tau\varphi(\gamma(\tau))\big)d\tau,
		\end{equation*}
		and
		\begin{equation*}
			T^{\varepsilon}_s\varphi(\gamma(s))=T^{\varepsilon}_{s-s_0}\circ T^{\varepsilon}_{s_0}\varphi(\gamma(s))\leq T^{\varepsilon}_{s_0}\varphi(\gamma(s_0))+\int_{s_0}^sL^{\varepsilon}\big(\gamma(\tau),\dot{\gamma}(\tau),T^{\varepsilon}_\tau\varphi(\gamma(\tau))\big)d\tau.
		\end{equation*}
		Thus we get
		$$F(s)\leq\int^s_{s_0}\Big|L^{\varepsilon}\big(\gamma(\tau),\dot{\gamma}(\tau),T^{\varepsilon}_\tau\varphi(\gamma(\tau))\big)-L\big(\gamma(\tau),\dot{\gamma}(\tau),T^-_\tau\varphi(\gamma(\tau))\big)\Big|d\tau.$$
		Using Proposition \ref{prop1} and \textbf{(L3)}, we obtain that
		\begin{equation*}
			\begin{split}
				F(s)
				\leq&\int^s_{s_0}\Big|L^{\varepsilon}\big(\gamma(\tau),\dot{\gamma}(\tau),T^{\varepsilon}_\tau\varphi(\gamma(\tau))\big)-L\big(\gamma(\tau),\dot{\gamma}(\tau),T^{\varepsilon}_\tau\varphi(\gamma(\tau))\big)\Big|d\tau\\
				+&\int^s_{s_0}\Big|L\big(\gamma(\tau),\dot{\gamma}(\tau),T^{\varepsilon}_\tau\varphi(\gamma(\tau))\big)-L\big(\gamma(\tau),\dot{\gamma}(\tau),T^-_\tau\varphi(\gamma(\tau))\big)\Big|d\tau\\
				\leq&\int_{s_{0}}^s \varepsilon+\int_{s_{0}}^s \Big|\frac{\partial L}{\partial u}(\gamma(\tau),\dot{\gamma}(\tau),\phi(\tau))\Big|\cdot\Big(T^{\varepsilon}_\tau\varphi(\gamma(\tau))-T^{-}_\tau\varphi(\gamma(\tau))\Big)d\tau,\\
				\leq&\varepsilon(s-s_{0})+\int_{s_{0}}^s\lambda F(\tau)d\tau, \quad \forall s\in(s_0,t].
			\end{split}
		\end{equation*}
		Next by Gronwall inequality, we have
		\[F(t)\leq \frac{\varepsilon}{\lambda}(e^{\lambda (t-s_{0})}-1)\leq\frac{\varepsilon}{\lambda}(e^{\lambda t}-1), \]
		that is
		$$T^{\varepsilon}_{t}\varphi(x)-T^{-}_{t}\varphi(x)\leq \frac{\varepsilon}{\lambda}(e^{\lambda t}-1),\quad \forall t>0,\ \forall x\in M.$$
		
		Now the proof is complete.
	\end{proof}

	\section{Proof of Theorem \ref{thm1}}

	\begin{proof}[ Proof of Theorem \ref{thm1}]
		Since the viscosity solution $u_-$ is locally asymptotic stable, there exists $\delta_0>0$ such that
		\begin{equation*}
			\lim_{t\to+\infty}\|T^{-}_{t}\varphi-u_{-}\|_{\infty}=0,
		\end{equation*}
		for any $\varphi\in C(M,R)$ satisfying $\|\varphi-u_{-}\|_{\infty}<\delta_0$.
		
		For all $\delta>0$, let $u_{\delta}=u_{-}-\min\{\delta,\delta_0\}$ and $u^{\delta}=u_{-}+\min\{\delta,\delta_0\}$ respectively. We note that $u_{-}$ is Lipschitz continuous on $M$, and  $u_{\delta},\ u^{\delta} $ are bounded. Thus we can find $t_{\delta}>0$ such that for all $t\geq t_{\delta}$,
		\begin{equation*}
			\|T^{-}_{t}u^{\delta}-u_{-}\|_{\infty}\leq\frac{1}{2}\min\{\delta,\delta_0\},
		\end{equation*}
		and
		\begin{equation*}
			\|T^{-}_{t}u_{\delta}-u_{-}\|_{\infty}\leq\frac{1}{2}\min\{\delta,\delta_0\}.
		\end{equation*}
		By Proposition \ref{prop2}, for any $\varepsilon>0$, we have $$\|T^{\varepsilon}_{t_{\delta}}u^{\delta}-T^{-}_{t_{\delta}}u^{\delta}\|_{\infty}\leq \frac{\varepsilon}{\lambda}(e^{\lambda t_{\delta}}-1),$$
		and
		\begin{equation*}\label{eq3.1}
			\|T^{\varepsilon}_{t_{\delta}}u_{\delta}-T^{-}_{t_{\delta}}u_{\delta}\|_{\infty}\leq \frac{\varepsilon}{\lambda}(e^{\lambda t_{\delta}}-1).
		\end{equation*}
		Then for any $\varepsilon\in[0,\frac{\lambda\min\{\delta,\delta_0\}}{2(e^{\lambda t_{\delta}}-1)}]$, we get
		$$\|T^{\varepsilon}_{t_{\delta}}u^{\delta}-u_{-}\|_{\infty}\leq\min\{\delta,\delta_0\} \quad \text{and} \quad \|T^{\varepsilon}_{t_{\delta}}u_{\delta}-u_{-}\|_{\infty}\leq\min\{\delta,\delta_0\}.$$
		So using  part (1) of Proposition \ref{prop01}, we have
		$$u^{\delta}=u_{-}+\min\{\delta,\delta_0\}\geq T^{\varepsilon}_{t_{\delta}}u^{\delta}\geq T^{\varepsilon}_{t_{\delta}}u_{\delta}\geq u_{-}-\min\{\delta,\delta_0\}=u_{\delta}.$$
		Therefore, using part (1) of Proposition  \ref{prop01}, we obtain for $n\in N$,
		\[u_{\delta}\leqslant T^{\varepsilon}_{t_{\delta}}u_{\delta}\leqslant T^{\varepsilon}_{2t_{\delta}}u_{\delta}\leqslant T^{\varepsilon}_{3t_{\delta}}u_{\delta}\leqslant \dots \leqslant T^{\varepsilon}_{nt_{\delta}}u_{\delta},\]
		and
		\[u^{\delta}\leqslant T^{\varepsilon}_{t_{\delta}}u^{\delta}\leqslant T^{\varepsilon}_{2t_{\delta}}u^{\delta}\leqslant T^{\varepsilon}_{3t_{\delta}}u^{\delta}\leqslant \dots \leqslant T^{\varepsilon}_{nt_{\delta}}u^{\delta}.\]
		Then we get 
		\begin{equation}\label{eq3.2}
			u^{\delta}\geq T^{\varepsilon}_{nt_{\delta}}u^{\delta}\geq T^{\varepsilon}_{nt_{\delta}}u_{\delta}\geq u_{\delta}, \quad \forall n\in \mathbf{N}.
		\end{equation}
		For any $t\in \R^{+}$, suppose $t=nt_{\delta}+s, n\in \mathbf{N}, s\in [0,t_{\delta})$. Then we can deduce that
		\begin{equation}\label{eq3.3}
			T^{\varepsilon}_su^{\delta}\geq T^{\varepsilon}_{t}u^{\delta}\geq T^{\varepsilon}_{t}u_-\geq T^{\varepsilon}_{t}u_{\delta}\geq T^{\varepsilon}_su_{\delta}, \quad \forall s\in[0,t_{\delta}).
		\end{equation}
		We note that $s<t_{\delta}$, $u_{\delta},\ u^{\delta}$ are bounded. The functions $(x,s)\to T^{\varepsilon}_su^{\delta}(x)$ and $(x,s)\to T^{\varepsilon}_su_{\delta}(x)$ are continuous on $M\times[0,t_{\delta}]$, $T^{\varepsilon}_su^{\delta}(x)$ and $T^{\varepsilon}_su_{\delta}(x)$  are bounded on $M\times[0,t_{\delta}]$, there is constant $C>0$ such that
		$$\big|T^{\varepsilon}_tu_-(x)\big|\leq C, \quad \forall (x,t)\in M\times[0,+\infty).$$		
		Denote by $l_1>0$ a Lipschitz constant of the function $(x_0,u_0,x)\to h^{\varepsilon}_{x_0,u_0}(x,1)$ on $M\times[-C,C]\times M$. Recall (4) of Proposition \ref{prop01}, for all $x,y\in M$, we have
		\begin{equation*}
			\begin{split}
				&\big|T^{\varepsilon}_{t}u_-(x)-T^{\varepsilon}_{t}u_-(y)\big|\\
				=&\big|\inf_{z\in M}h^{\varepsilon}_{z,T^{\varepsilon}_{t-1}u_-(z)}(x,1)-\inf_{z\in M}h^{\varepsilon}_{z,T^{\varepsilon}_{t-1}u_-(z)}(y,1)\big|\\
				\leq&\sup_{z\in M}\big|h^{\varepsilon}_{z,T^{\varepsilon}_{t-1}u_-(z)}(x,1)-h^{\varepsilon}_{z,T^{\varepsilon}_{t-1}u_-(z)}(y,1)\big|\\
				\leq&l_{1}d(x,y), \quad \forall t>1.		 		
			\end{split}
		\end{equation*}
		So $\{T^{\varepsilon}_tu_-(x)\}_{t>1}$ is equi-Lipschitz on $M$. Let $u^{\varepsilon}_{-}(x):=\liminf_{t\to +\infty}T^{\varepsilon}_{t}u_-(x)$. It is easy to prove that
		$$u^{\varepsilon}_-(x)=\lim_{t\to+\infty}\inf_{s\geqslant t}T^{\varepsilon}_su_-(x) \text{ uniformly on } x\in M.$$
		Note that
		$$u^{\varepsilon}_{-}=\lim_{\sigma\to +\infty}\inf_{s\geqslant \sigma}T^{\varepsilon}_{t}\circ T^{\varepsilon}_{s}u_-=\lim_{\sigma\to +\infty}T^{\varepsilon}_{t}(\inf_{s\geqslant \sigma}T^{\varepsilon}_{s}u_-)=T^{\varepsilon}_{t}(\lim_{\sigma\to +\infty}\inf_{s\geqslant \sigma}T^{\varepsilon}_{s}u_-)=T^{\varepsilon}_{t}u^{\varepsilon}_{-},\quad \forall t>0.$$
		It shows that $u^{\varepsilon}_-$ is a fixed point of $\{T^{\varepsilon}_t\}_{t\geqslant0}$. By Proposition \ref{prop9}, $u^{\varepsilon}_-$ is a viscosity solution of \eqref{eqth1}.
		
		On the other hand, from Proposition \ref{prop9} and Proposition \ref{prop2}, when $\varepsilon\leq \frac{\lambda\min\{\delta,\delta_0\}}{2(e^{\lambda t_{\delta}}-1)}$, we have
		\begin{equation}\label{eq3.4}
			T^{\varepsilon}_{s}u_-\geq T^{-}_{s}u_--\frac{\varepsilon}{\lambda}(e^{\lambda s}-1)\geq u_--\frac{1}{2}\min\{\delta,\delta_0\}>u_{\delta}, \quad \forall s\in[0,t_{\delta}].
		\end{equation}
	
		For any $t\in \R^{+}$, suppose $t=nt_{\delta}+s, n\in \mathbf{N}, s\in [0,t_{\delta})$. Then by \eqref{eq3.2} and \eqref{eq3.4}, we can deduce that
		\begin{equation}\label{eq3.5}
			T^{\varepsilon}_tu_-\geq T^{\varepsilon}_{nt_{\delta}}u_{\delta}\geq u_{\delta}.
		\end{equation}
		Thus according to formulas \eqref{eq3.2}, \eqref{eq3.3} and \eqref{eq3.5}, we get
		$$u_-+\delta\geq u^{\delta}\geq\liminf_{n\to+\infty}T^{\varepsilon}_{nt_{\delta}}u^{\delta}\geq \liminf_{t\to+\infty}T^{\varepsilon}_{t}u^{\delta}\geq \liminf_{t\to+\infty}T^{\varepsilon}_{t}u_-\geq u_{\delta}\geq u_--\delta,$$
		which means $$\|u^{\varepsilon}_--u_-\|\leq\delta, \quad \forall \varepsilon\leq\varepsilon_{\delta}:=\frac{\lambda\min\{\delta,\delta_0\}}{2(e^{\lambda t_{\delta}}-1)}.$$
	\end{proof}

	\section{Proof of Theorem \ref{thm2}}

	Recall the following result.
	
	\begin{lemma}(\cite[Lemma 2.3]{Z24})\label{lemma4.1}
		Let $u_-\in\mathcal{S}^-$. If for all $(x,p,u)\in\Lambda_{u_-}$, $\frac{\partial H}{\partial u}(x,p,u)>0$,  then there exists $\delta_0>0$, such that for any $\varphi\in C(M,\R)$ satisfying $\|\varphi-u_-\|_{\infty}<\delta_0$, the following holds
		$$\lim_{t\to+\infty}T^-_t\varphi=u_-.$$
	\end{lemma}

	In the light of Lemma \ref{lemma4.1}, we know that $u_-\in\mathcal{S}^-$ will be locally asymptotically stable if  for all $(x,p,u)\in\Lambda_{u_-}$, $u_-$ satisfies $\frac{\partial H}{\partial u}(x,p,u)>0$.  To finish the proof of Theorem \ref{thm2}, we still need the following result.

	\begin{lemma}\label{prop5}
		Let $u_-\in\mathcal{S}^-$. If $\frac{\partial H}{\partial u}(x,p,u)>0$, $\forall(x,p,u)\in\Lambda_{u_-}$, then there is constant $\varepsilon_0$ such that for any $\varepsilon\in[0,\varepsilon_0]$, the equation \eqref{eqth1} has a viscosity solution $u^{\varepsilon}_-$ satisfying
		$$\frac{\partial H^{\varepsilon}}{\partial u}(x,p,u)>0, \quad \forall (x,p,u)\in\Lambda_{u^{\varepsilon}_-}.$$
	\end{lemma}

	\begin{proof}
		From Lemma \ref{lemma4.1} and Theorem \ref{thm1}, we can find that for any given  $\delta\in(0,1)$, and for all $\varepsilon\in[0,\varepsilon_{\delta}]$, the equation \eqref{eqth1} has at least one viscosity solution $u^{\varepsilon}_{-}$ satisfying
		\begin{equation}\label{eq4.4}
			\|u^{\varepsilon}_{-}-u_{-}\|\leqslant\delta<1,
		\end{equation}
		where $\varepsilon_{\delta}$ is a positive constant depending on $\delta$.
		
		Since the rest of proof is quite long, we divide it into three steps.
		
		\item [\bf{Step 1:}] We assert that for any $\varepsilon\in[0,\min\{1,\varepsilon_{\delta}\}]$ and any $(u^{\varepsilon}_-,L^{\varepsilon},0)$-calibrated curve $\gamma_{\varepsilon}:[a,b]\to M$, there is some constant $C$ depending only on $u_-$ and the Lagrangian $L$, such that
		$$\|\dot{\gamma}_{\varepsilon}(s)\|\leqslant C, \quad \forall s\in(a,b).$$

		In fact,  as a viscosity solution of \eqref{eqth1}, $u^{\varepsilon}_-$ is dominated by $L^{\varepsilon}$. Thus by Proposition \ref{prop1}, for all $\varepsilon\in[0,\min\{1,\varepsilon_{\delta}\}]$ and each continuous piecewise $C^1$ curve $\gamma:[t_1,t_2]\rightarrow M$, we have
		\begin{equation}\label{eq4.1}
			\begin{split}
				u^{\varepsilon}_-(\gamma(t_2))-u^{\varepsilon}_-(\gamma(t_1))&\leqslant\int_{t_1}^{t_2}L^{\varepsilon}(\gamma(s),\dot{\gamma}(s),u^{\varepsilon}_-(\gamma(s)))ds\\
				&\leqslant\int_{t_1}^{t_2}\Big(L(\gamma(s),\dot{\gamma}(s),u^{\varepsilon}_-(\gamma(s)))+1\Big)ds,
			\end{split}
		\end{equation}
		 Note that $u^{\varepsilon}_-\in[\min_{x\in M}u_--1,\max_{x\in M}u_-+1]$, we denote
		$$A:=\sup\Big\{L(x,v,u)+1 \big|\ (x,v,u)\in TM\times\R,\ \|v\|_x=1,\ u\in[\min_{x\in M}u_--1,\max_{x\in M}u_-+1]\Big\}.$$
		Then from \eqref{eq4.1}, it is easy to find that
		$$|u^{\varepsilon}_-(x)-u^{\varepsilon}_-(y)|\leqslant A\cdot d(x,y), \quad \forall x,\ y\in M,\ \forall \varepsilon\in[0,\min\{1,\varepsilon_{\delta}\}].$$
		So $u^{\varepsilon}_-$ is Lipschitz. Take a $(u^{\varepsilon}_-,L^{\varepsilon},0)$-calibrated curve $\gamma_{\varepsilon}:[a,b]\to M$. By the definition of calibrated curve, 	then	for all $a\leqslant t_1<t_2\leqslant b$, we have
		\begin{equation}\label{eq4.2}
			A\cdot d(\gamma_{\varepsilon}(t_1),\gamma_{\varepsilon}(t_2))\geqslant u^{\varepsilon}_-(\gamma_{\varepsilon}(t_2))-u^{\varepsilon}_-(\gamma_{\varepsilon}(t_1))=\int_{t_1}^{t_2}L^{\varepsilon}(\gamma_{\varepsilon}(s),\dot{\gamma}_{\varepsilon}(s),u^{\varepsilon}_-(\gamma_{\varepsilon}(s)))ds,
		\end{equation}
 Using Proposition \ref{prop1}, \textbf{(L2)-(L3)}  and \eqref{eq4.4}, for some constant $C_{u_-,A}$ depending on $A$ and $u_-$, we get
		\begin{equation}\label{eq4.3}
			\begin{split}
				&L^{\varepsilon}(\gamma_{\varepsilon}(s),\dot{\gamma}_{\varepsilon}(s),u^{\varepsilon}_-(\gamma_{\varepsilon}(s)))\\
				\geqslant& L(\gamma_{\varepsilon}(s),\dot{\gamma}_{\varepsilon}(s),u^{\varepsilon}_-(\gamma_{\varepsilon}(s)))-1\\
				\geqslant& L(\gamma_{\varepsilon}(s),\dot{\gamma}_{\varepsilon}(s),u_-(\gamma_{\varepsilon}(s)))-\lambda\cdot(u^{\varepsilon}_-(\gamma_{\varepsilon}(s))-u_-(\gamma_{\varepsilon}(s)))-1\\
				\geqslant& (A+1)\|\dot{\gamma}_{\varepsilon}(s)\|-C_{u_-,A}-\lambda-1, \quad \forall s\in[a,b],
			\end{split}
		\end{equation}
		 According to \eqref{eq4.2} and \eqref{eq4.3}, we deduce that
		$$d(\gamma_{\varepsilon}(t_1),\gamma_{\varepsilon}(t_2))\leqslant(C_{u_-,A}+\lambda+1)\cdot(t_2-t_1), \quad for\ any\ a\leqslant t_1<t_2\leqslant b.$$
		From Proposition \ref{diffe}, we know $\gamma_{\varepsilon}$ is differentiable in $(a,b)$. Therefore,
		$$\|\dot{\gamma}_{\varepsilon}(s)\|\leqslant C_{u_-,A}+\lambda+1, \quad \forall s\in(a,b).$$
		Thus the assertion holds true.
		
		\item [\bf{Step 2:}] We will prove the following assertion.
		
		For any $\delta>0$, there is a constant $\tilde{\varepsilon}_{\delta}>0$ such that for all $\varepsilon\in[0,\tilde{\varepsilon}_{\delta}]$, we have
		$$dist((x,p,u),\Lambda_{u_-})\leqslant\delta, \quad \forall (x,p,u)\in\Lambda_{u^{\varepsilon}_-}$$
		where $dist(\alpha,S):=\inf_{\beta\in S}d(\alpha,\beta)$ for all $\alpha\in T^*M\times\R$ and $S\subseteq T^*M\times\R$.
		
		Assume by contradiction that there are a constant $c_0>0$, a sequence $\{\varepsilon_n\}$ with $\varepsilon_n\to0$ as $n\to+\infty$ and a sequence $(x_n,p_{n},u_n)\in\Lambda_{u^{\varepsilon_n}_-}$ such that
		\begin{equation}\label{eq4.6}
			dist((x_n,p_n,u_n),\Lambda_{u_-})\geqslant c_0, \quad \forall n\in\mathbf{N}^+.
		\end{equation}
		By Proposition \ref{prop2.7}, we can find $(u^{\varepsilon_n}_-,L,0)$-calibrated curve $\gamma_n:[-1,0]\to M$ and $t_n\in(-1,0)$ for each $x_n$ such that
		\begin{equation}\label{eq4.5}
			d((x_n,p_n,u_n),(\g_n(t_n),\tilde{p}_n(t_n),\tilde{u}_n(t_n)))<\frac{c_0}{3}, \quad \forall n\in\mathbf{N}^+,
		\end{equation}
		where
		$$\tilde{p}_n(s):=D_xu^{\varepsilon_n}_-(\g_n(s)), \quad \tilde{u}_n(s):=u^{\varepsilon_n}_-(\g_n(s)), \quad \forall s\in[-1,0].$$
		Based on Step 1, $\{\dot{\g}_n\}$ and $\{\tilde{p}_n\}$ are uniformly bounded. From Proposition \ref{inva}, we know $(\g_n(s),\tilde{p}_n(s),\tilde{u}_n(s))\ (s\in[-1,0])$ is the flow generated by system \begin{align*}
			\left\{
			\begin{array}{l}
				\dot{x}=\frac{\partial H^{\varepsilon_n}}{\partial p}(x,p,u),\\[1mm]
				\dot{p}=-\frac{\partial H^{\varepsilon_n}}{\partial x}(x,p,u)-\frac{\partial H^{\varepsilon_n}}{\partial u}(x,p,u)p,\qquad (x,p,u)\in T^*M\times\mathbf{R},\\[1mm]
				\dot{u}=\frac{\partial H^{\varepsilon_n}}{\partial p}(x,p,u)\cdot p-H^{\varepsilon_n}(x,p,u).
			\end{array}
			\right.
		\end{align*}
		So
		\[
		\dot{\gamma}_n(s)=\frac{\partial H^{\varepsilon_n}}{\partial p}\big(\gamma_n(s),\tilde{p}_n(s),\tilde{u}_n(s)\big),
		\]
		and
		\begin{equation*}
			\begin{split}
				&\ddot{\gamma}_n=\frac{\partial^2H^{\varepsilon_n}}{\partial p\partial x}\big(\gamma_n,\tilde{p}_n,\tilde{u}_n\big)\cdot\dot{\gamma}_n+\frac{\partial^2H^{\varepsilon_n}}{\partial p\partial u}\big(\gamma_n,\tilde{p}_n,\tilde{u}_n\big)\cdot\dot{\tilde{u}}_n+\frac{\partial^2H^{\varepsilon_n}}{\partial p^2}\big(\gamma_n,\tilde{p}_n,\tilde{u}_n\big)\cdot\dot{\tilde{p}}_n\\
				=&\Big\{\frac{\partial^2H^{\varepsilon_n}}{\partial p\partial x}\big(\gamma_n,\tilde{p}_n,\tilde{u}_n\big)\cdot\frac{\partial H^{\varepsilon_n}}{\partial p}\big(\gamma_n,\tilde{p}_n,\tilde{u}_n\big)+\frac{\partial^2H^{\varepsilon_n}}{\partial p\partial u}\big(\gamma_n,\tilde{p}_n,\tilde{u}_n\big)\cdot\big(\frac{\partial H^{\varepsilon_n} }{\partial p}\big(\gamma_n,\tilde{p}_n,\tilde{u}_n\big)\cdot \tilde{p}_n\\
				-&H^{\varepsilon_n}\big(\gamma_n,\tilde{p}_n,\tilde{u}_n\big)\big)+\frac{\partial^2H^{\varepsilon_n}}{\partial p^2}\big(\gamma_n,\tilde{p}_n,\tilde{u}_n\big)\cdot\big(-\frac{\partial H^{\varepsilon_n}}{\partial x}\big(\gamma_n,\tilde{p}_n,\tilde{u}_n\big)-\frac{\partial H^{\varepsilon_n}}{\partial u}\big(\gamma_n,\tilde{p}_n,\tilde{u}_n\big)\cdot \tilde{p}_n\big)\Big\}.
			\end{split}
		\end{equation*}
		Therefore, $\{\ddot{\g}_n\}$ are also uniformly bounded. By Arzel$\grave{a}$-Ascoli Theorem and if necessary passing to a subsequence, we may assume that $\g_n\to \g$ and $\dot{\g}_n\to \dot{\g}$ uniformly on $[-1,0]$ for some curve $\gamma$. Since $u^{\varepsilon_n}_-\to u_-$ uniformly on $M$ as $n\to+\infty$ and $\gamma_n$ satisfies
		$$u_{-}^{\varepsilon_n}(\gamma_n(b))-u_{-}^{\varepsilon_n}(\gamma_n(a))=\int^{b}_{a}L^{\varepsilon_n}(\gamma_n(s),\dot{\gamma}_n(s),u^{\varepsilon_n}_-(\gamma_n(s)))ds, \quad -1\leqslant a<b\leqslant0,$$
		we can deduce that
		$$u_{-}(\gamma(b))-u_{-}(\gamma(a))=\int^{b}_{a}L(\gamma(s),\dot{\gamma}(s),u_-(\gamma(s)))ds, \quad -1\leqslant a<b\leqslant0.$$
		That means $\g$ is a $(u_-,L,0)$-calibrated curve. From Proposition \ref{inva},
		$$\tilde{p}_n(s)=D_xu^{\varepsilon_n}_-(\g_n(s))=\frac{\partial L^{\varepsilon_n}}{\partial \dot{x}}(\g_n(s),\dot{\g}_n(s),u^{\varepsilon_n}_-(\g_n(s))), \quad \forall s\in(-1,0),\ \forall n\in\mathbf{N}^+.$$
		Note that
		$$\frac{\partial L^{\varepsilon_n}}{\partial\dot{x}}\to\frac{\partial L}{\partial\dot{x}} \text{ uniformly on } TM\times\R \text{ and } u^{\varepsilon_n}_-\to u_- \text{ uniformly on } M, \text{ as } n\to+\infty.$$
		Then we have $\tilde{p}_n(s)\to\frac{\partial L}{\partial\dot{x}}(\g(s),\dot{\g}(s),u_-(\g(s)))=D_xu_-(\g(s))$ uniformly on $(-1,0)$ as $n\to+\infty$. So
		$$(\g_n(s),\tilde{p}_n(s),\tilde{u}_n(s))\to(\g(s),D_xu_-(\g(s)),u_-(\g(s)))\in\Lambda_{u_-}, \quad \forall s\in(-1,0).$$
		Because $\Lambda_{u_-}$ is closed, there is a constant $N>0$ such that
		$$dist((\g_n(t_n),\tilde{p}_n(t_n),\tilde{u}_n(t_n)),\Lambda_{u_-})<\frac{c_0}{3}, \quad \forall n>N.$$
		Recall \eqref{eq4.5}, we have
		\begin{equation*}
			\begin{split}
				&dist((x_n,p_n,u_n),\Lambda_{u_-})\\
				\leqslant& d((x_n,p_n,u_n),(\g_n(t_n),\tilde{p}_n(t_n),\tilde{u}_n(t_n)))+dist((\g_n(t_n),\tilde{p}_n(t_n),\tilde{u}_n(t_n)),\Lambda_{u_-})\\
				\leqslant&\frac{c_0}{3}+\frac{c_0}{3}<c_0.
			\end{split}
		\end{equation*}
		which contradicts \eqref{eq4.6}. Therefore, the assertion holds true.
		
		\item [\bf{Step 3:}] Since $\frac{\partial H}{\partial u}(x,p,u)>0,\ \forall(x,p,u)\in\Lambda_{u_-}$, from Step 2, we can find constant $\varepsilon_1>0$ such that for any $\varepsilon\in[0,\varepsilon_1]$,
		$$\frac{\partial H}{\partial u}(x,p,u)>0,\ \forall(x,p,u)\in\Lambda_{u^{\varepsilon}_-}.$$
		On the other hand, we get
		$$\Big|\frac{\partial H}{\partial u}(x,p,u)-\frac{\partial H^{\varepsilon}}{\partial u}(x,p,u)\Big|\leqslant\varepsilon\cdot\Big|\frac{\partial P}{\partial u}(x,p,u)\Big|, \quad \forall(x,p,u)\in T^*M\times\R.$$
		So there is $\varepsilon_0>0$ such that for any $\varepsilon\in[0,\varepsilon_0]$,
		$$\frac{\partial H^{\varepsilon}}{\partial u}(x,p,u)>0,\ \forall(x,p,u)\in\Lambda_{u^{\varepsilon}_-}.$$
		
		Now the proof is complete.
	\end{proof}

	\begin{proof}[ Proof of Theorem \ref{thm2}]
		From Lemma \ref{lemma4.1} and Theorem \ref{thm1}, we know the result of Theorem \ref{thm1} holds. And from Lemma \ref{prop5}, we know that there exists $\delta_0'>0$, such that for any $\varphi\in C(M,\R)$ satisfying $\|\varphi-u^{\varepsilon}_-\|_{\infty}<\delta_0'$, the following holds
		$$\lim_{t\to+\infty}T^-_t\varphi=u^{\varepsilon}_-.$$
		That means $u^{\varepsilon}_-$ is locally asymptotically stable and is the unique viscosity solution of equation \eqref{eqth1} satisfying $$\|u^{\varepsilon}_--u_-\|\leqslant\delta_0, \quad \text{for some }\delta_0>0.$$
		Combining the result of Theorem \ref{thm1} and Lemma \ref{lemma4.1}, we get the result of Theorem \ref{thm2}.
	\end{proof}

\textbf{Acknowledgements}  We would like to thank the referees for their valuable comments,which have significantly
improved the presentation of the results in this paper.\\

\textbf{Data availability} Data sharing not applicable to this article as no datasets were generated or analyzed in this study.\\

\textbf{Declarations}

\textbf{Conflict of interest} No conflict of interest exits in the submission of this manuscript, and manuscript is
approved by all authors for publication.

	\bibliographystyle{plain}

\end{document}